\documentclass[a4paper,12pt,oneside]{amsart}
\pdfoutput=1
\usepackage{microtype}
\usepackage{amssymb,latexsym}
\usepackage[margin=3.25cm]{geometry}
\usepackage{mathtools}
\usepackage{amsthm}
\usepackage[inline]{enumitem}
\usepackage{mathrsfs}
\usepackage{mleftright}
\usepackage[dvipsnames]{xcolor}
\usepackage{hyperref}
\hypersetup{
    colorlinks=true,
    urlcolor=black,
    linkcolor=black, 
    citecolor=blue} 

\DeclareMathOperator{\Rank}{rank}
\DeclareMathOperator{\Span}{span}

\newcommand{\diff}{\mathop{}\!d}

\protected\def\verythinspace{%
  \ifmmode
    \mskip0.5\thinmuskip
  \else
    \ifhmode
      \kern0.08334em
    \fi
  \fi
}

\theoremstyle{plain}
\newtheorem{theorem}{Theorem}[section]
\newtheorem{proposition}[theorem]{Proposition}
\newtheorem{lemma}[theorem]{Lemma}

\theoremstyle{remark}
\newtheorem{remark}[theorem]{Remark}
\newtheorem*{remark*}{Remark}
\newtheorem*{remarks*}{Remarks}
\theoremstyle{definition}

\newtheorem{definition}[theorem]{Definition}
\newtheorem*{notation*}{Notation}
\newtheorem*{example*}{Example}

\begin{document}
\title[Rank-one submanifolds]{Classification of rank-one submanifolds}
\author{Matteo Raffaelli}
\address{Institute of Discrete Mathematics and Geometry\\
TU Wien\\
Wiedner Hauptstra{\ss}e 8-10/104\\
1040 Vienna\\
Austria}
\email{matteo.raffaelli@tuwien.ac.at}
\thanks{This work was supported by Austrian Science Fund (FWF) project F~77 (SFB ``Advanced Computational Design'').}
\date{August 21, 2023}
\begin{abstract}
We study ruled submanifolds of Euclidean space. First, to each (parametrized) ruled submanifold $\sigma$, we associate an integer-valued function, called \textit{degree}, measuring the extent to which $\sigma$ fails to be cylindrical. In particular, we show that if the degree is constant and equal to $d$, then the singularities of $\sigma$ can only occur along an $(m-d)$-dimensional ``striction'' submanifold. This result allows us to extend the standard classification of developable surfaces in $\mathbb{R}^{3}$ to the whole family of flat and ruled submanifolds without planar points, also known as \textit{rank-one}: an open and dense subset of every rank-one submanifold is the union of \emph{cylindrical}, \emph{conical}, and \emph{tangent} regions.
\end{abstract}
\maketitle
\tableofcontents

\section{Introduction and main result}
Developable surfaces in $\mathbb{R}^{3}$ are classical objects in differential geometry, enjoying a rich collection of properties. For instance, they have zero Gaussian curvature; they are envelopes of one-parameter families of planes; they are ruled, with tangent plane stable along the rulings.

Developable surfaces started to be studied in depth during the eighteenth century (even before the term ``differential geometry" was coined~\cite{lawrence2011}), when Leonhard Euler and Gaspard Monge got interested in their properties and classification. Their efforts led to the following, nowadays well-known, theorem.

\begin{theorem}[{\cite[section~3.24]{kuehnel2015}}]\label{DevelopablesClassification}
An open and dense subset of every developable surface is a union of planar regions, cones, cylinders, and tangent surfaces of space curves.
\end{theorem}

In arbitrary dimension and codimension, a broad generalization of the notion of developable surface (without planar points) is that of \emph{rank-one submanifold}. Let $M$ be an $m$-dimensional submanifold of $\mathbb{R}^{m+n}$. We say that $M$ is \textit{rank-one} if the rank of the Gauss map $M \to G(m, \mathbb{R}^{m+n})$ is constant and equal to one; equivalently, if $M$ is both flat and $(m-1)$-ruled.

Rank-one---and, more generally, constant-rank---submanifolds received a great deal of attention during the 1960s~\cite[p.~ix]{rovenskii1998}, yet they continue to be the subject of intensive study; see, e.g., \cite{raffaelli202x, onti2020, florit2017, vittone2012, dajczer2009, dajczer2004}.

It is therefore natural to ask whether a statement similar to that of Theorem~\ref{DevelopablesClassification} might hold for rank-one submanifolds~\cite{orourke2011}. In this paper, we answer such question affirmatively, by showing that there are only three types of rank-one submanifolds: \emph{cylindrical}, \emph{conical}, and \emph{tangent}.

\begin{definition}[Cylinder]\label{CylinderDEF}
Given a smooth unit-speed curve $\gamma \colon I \to \mathbb{R}^{m+n}$ and an $(m-1)$-tuple of orthonormal vectors $X_{1}, \dotsc, X_{m-1}$, the map
\begin{align*}
I \times \mathbb{R}^{m-1} &\to \mathbb{R}^{m+n}\\
(t, u^{1}, \dotsc, u^{m-1}) &\mapsto \gamma(t) + u^{1}X_{1} + \dotsb + u^{m-1} X_{m-1}
\end{align*}
is called a \textit{cylindrical submanifold} or, more simply, a \textit{cylinder}.
\end{definition}

\begin{definition}[Cone]
Let $U \subset \mathbb{R}^{m-1}$, and let $\psi \colon U \to \mathbb{R}^{m+n}$ be a smooth map whose differential has rank $m-2$ everywhere. Given a smooth unit vector field $X_{\psi}$ along $\psi$ such that
\begin{equation*}
X_{\psi} \notin \Span \mleft( \dfrac{\partial \psi}{\partial u^{1}}, \dotsc, \dfrac{\partial \psi}{\partial u^{m-1}}\mright),
\end{equation*}
the map
\begin{align*}
U \times \mathbb{R} &\to \mathbb{R}^{m+n}\\
(u^{1}, \dotsc, u^{m}) &\mapsto \psi(u^{1}, \dotsc, u^{m-1}) + u^{m}X_{\psi}(u^{1}, \dotsc, u^{m-1})
\end{align*}
is called a \textit{conical submanifold} or, more simply, a \textit{cone}.
\end{definition}

\begin{definition}[Tangent submanifold]
Let $\xi \colon U \to \mathbb{R}^{m+n}$ be an $(m-1)$-dimensional  \emph{regular} submanifold of $\mathbb{R}^{m+n}$. Given a smooth unit vector field $X_{\xi}$ along $\xi$ such that
\begin{equation*}
	\dfrac{\partial \xi}{\partial u^{1}} \wedge \dotsb \wedge \dfrac{\partial \xi}{\partial u^{m-1}} \wedge X_{\xi} = 0,
\end{equation*}
	the map
\begin{align*}
U \times \mathbb{R} &\to \mathbb{R}^{m+n}\\
(u^{1}, \dotsc, u^{m}) &\mapsto \xi(u^{1}, \dotsc, u^{m-1}) + u^{m}X_{\xi}(u^{1}, \dotsc, u^{m-1})
\end{align*}
is called a \textit{tangent submanifold of $\xi$}.
\end{definition}



\begin{theorem}\label{MainTHM}
An open and dense subset of every rank-one submanifold is a union of cylindrical, conical, and tangent submanifolds. Conversely, any cylindrical and any singular ruled submanifold of degree one, as defined in section \ref{Preliminaries}, is rank-one.
\end{theorem}

We emphasize that Theorem~\ref{MainTHM} does not involve any assumption on the codimension of $M$. On the other hand, it is worth noting that when $M$ is rank-one, the first normal space (span of the image of the second fundamental form) is everywhere of dimension one; hence a rank-one submanifold is necessarily the image of an isometric composition~\cite{dajczer1992}, as explained by the following theorem.

\begin{theorem}[\cite{erbacher1971, dajczer1994}] 
Suppose that $M$ is rank-one. If the first normal bundle is parallel in the normal connection, then $M$ is contained in an $(m+1)$-dimensional totally geodesic submanifold of $\mathbb{R}^{m+n}$. On the other hand, if the first normal bundle is not parallel, then $M$ is contained in a rank-one hypersurface of $\mathbb{R}^{m+n}$ as a totally geodesic submanifold.
\end{theorem}

Recall that rank-one hypersurfaces can be locally described in terms of the Gauss parametrization; see \cite{dajczer1985}.

The rest of the paper, which is devoted to proving Theorem~\ref{MainTHM}, is organized as follows. After discussing some background material, in section~\ref{TheStrictionSubmanifold} we generalize the classic theory of noncylindrical ruled surfaces in $\mathbb{R}^{3}$ to arbitrary dimension and codimension; in particular, we verify that, under some reasonable assumption, the singularities of a ruled submanifold always concentrate along a lower dimensional \emph{striction submanifold}. In section~\ref{Rank-oneSubmanifolds} we then specialize our theory to the rank-one case and show that a noncylindrical rank-one submanifold is necessarily singular. Finally, by applying the results obtained earlier, in section~\ref{ProofMainTHM} we prove Theorem~\ref{MainTHM}.

The classification problem for rank-one submanifold was also studied by Ushakov in his Ph.D.\ thesis~\cite{ushakov1993}. The author came up with a classification theorem~\cite[Theorem~2]{ushakov1993} that looks quite similar to ours (although the techniques used in the proof do not). Unfortunately we were not able to fully understand such result and therefore cannot provide a comparison with Theorem~\ref{MainTHM}.

\section{Preliminaries}\label{Preliminaries}
In this section we discuss some preliminaries. Standard references for submanifolds are \cite{dajczer2019, berndt2016} and \cite[Chapter~8]{lee2018}.

\subsection{Distribution along curves}
Let $\gamma \colon I \to \mathbb{R}^{m+n}$ be a smooth regular curve. Without loss of generality, we may assume $\gamma$ to be unit-speed. A \textit{distribution of rank $r$ along $\gamma$} is a rank-$r$ subbundle of the ambient tangent bundle $T\mathbb{R}^{m+n} \rvert_{\gamma}$ over $\gamma$. Recall that
\begin{equation*}
	T\mathbb{R}^{m+n}\lvert_{\gamma} \:= \bigsqcup_{t \,\in \, I} T_{\gamma(t)}\mathbb{R}^{m+n}.
\end{equation*}

Let $\mathcal{D}$ be a distribution of rank $r$ along $\gamma$, and let $\mathcal{D}^{\perp}$ be the distribution of rank $m+n-r$ along $\gamma$ whose fiber at $t$ is the orthogonal complement of $\mathcal{D}_{t}$ in $\mathbb{R}^{m+n}$. For each $t \in I$, define a map $\rho_{t} \colon \mathcal{D}_{t} \to \mathcal{D}^{\perp}_{t}$ by
\begin{equation*}
v \mapsto \pi^{\perp}\frac{\diff v}{\diff t}(t),
\end{equation*}
where $v$ is extended arbitrarily to a smooth local section of $\mathcal{D}$, and where $\pi^{\perp}$ denotes orthogonal projection onto $\mathcal{D}^{\perp}$.

\begin{definition} 
The rank of $\rho_{t}$ is called the \textit{degree} of the distribution $\mathcal{D}$ at $t$. We say that $\mathcal{D}$ is of \textit{degree} $d$ if $d_{t} \coloneqq \Rank \rho_{t} = d$ for all $t \in I$.
\end{definition}

\begin{remark}\label{DegreeInequality}
Note that $d_{t} \leq \min(r, m+n-r)$.	
\end{remark}

Now, let $(E_{1}, \dotsc, E_{r})$ be a frame for $\mathcal{D}$, i.e., an $r$-tuple of smooth vector fields along $\gamma$ forming a basis of $\mathcal{D}_{t}$ for all $t \in I$. In particular, suppose that $(E_{1}, \dotsc, E_{r})$ is parallel with respect to the induced connection on $\mathcal{D}$, so that $\pi^{\top}(\dot{E}_{1}) = \dotsb = \pi^{\top}(\dot{E}_{r}) = 0$; here $\pi^{\top}$ is the orthogonal projection onto $\mathcal{D}$. Then 
\begin{equation*}
\Rank \rho_{t} = \dim \Span( \dot{E}_{1}(t), \dotsc, \dot{E}_{r}(t)).
\end{equation*}

\subsection{Ruled submanifolds}
A \textit{(parametrized) submanifold of $\mathbb{R}^{m+n}$} is a smooth map $V \to \mathbb{R}^{m+n}$, where $V$ is a subset of $\mathbb{R}^{m}$. A point where the differential is injective is called \textit{regular}; otherwise it is \textit{singular}. A submanifold is itself called \textit{regular} if all its points are regular, and \textit{singular} otherwise.

\begin{remark}
	A submanifold, even when regular, may have self-intersections in its image.
\end{remark}

\begin{definition} 
Given a smooth unit-speed curve $\gamma \colon I \to \mathbb{R}^{m+n}$ and a smooth orthonormal frame $(X_{j})_{j=1}^{m-1}$ along $\gamma$, the submanifold
\begin{align*}
	\sigma \colon	I \times \mathbb{R}^{m-1} &\to \mathbb{R}^{m+n}\\
		(t, u^{1}, \dotsc, u^{m-1}) &\mapsto \gamma(t) + u^{1} X_{1}(t) + \dotsb + u^{m-1} X_{m-1}(t)
	\end{align*}
is called a \textit{ruled submanifold} of $\mathbb{R}^{m+n}$. We say that $\gamma$ is a \textit{directrix} of $\sigma$. Moreover, for any fixed $t$, the $(m-1)$-dimensional affine subspace  of $\mathbb{R}^{m+n}$ spanned by $X_{1}(t), \dotsc, X_{m-1}(t)$ is called a \textit{ruling of $\sigma$}.
\end{definition}

The simplest examples of ruled submanifolds are the cylinders (Definition~\ref{CylinderDEF}). A cylinder is a ruled submanifold whose rulings are all parallel (in the usual Euclidean sense). On the other hand, a ruled submanifold is said to be \textit{noncylindrical} if the degree of the ruling distribution $\Span (X_{j})_{j = 1}^{m-1}$ is nonzero for all $t \in I$. Similarly, we speak of ruled submanifolds of \textit{degree} $d$ if the degree of the ruling distribution is constant and equal to $d$.

Recall that the first normal space of a Riemannian submanifold at a point $p$ is the linear subspace of the normal space spanned by the image of the second fundamental form at $p$. For a ruled submanifold, the degree of the ruling distribution and the dimension of the first normal space are closely related, as the next lemma shows.

\begin{lemma}
Let $\sigma$ be a ruled submanifold. If $\sigma$ is of degree $d$, then, at any point $p \in \sigma$, the first normal space $N_{p}^{1}\sigma$ satisfies
\begin{equation*}
d-1 \leq \dim N_{p}^{1}\sigma \leq d+1.
\end{equation*}
\end{lemma}

\begin{proof}
Let $X_{0} = \dot{\gamma}$, and let $\mathit{II}$ be the second fundamental form of $\sigma$. Since $\mathit{II}$ vanishes on each ruling, the first normal space is spanned by
\begin{equation*}
\mathit{II}(x_{0}, x_{0}), \dotsc, \mathit{II}(x_{0}, x_{m-1});
\end{equation*}
here $x_{0} = X_{0} \rvert_{p}$,  $x_{1} = X_{1} \rvert_{p}$, etc.
Recall that
\begin{equation*}
\mathit{II}(X_{0}, X_{j}) = \pi^{\perp}_{\sigma}\dot{X}_{j} = \pi^{\perp}_{\sigma}\rho(X_{j}),
\end{equation*}
where $\pi^{\perp}_{\sigma}$ denotes orthogonal projection onto $N\sigma$. 

Assume that $\sigma$ is of degree $d$. Since
\begin{equation*}
\Span \mleft(\rho_{t}(x_{1}), \dotsc, \rho_{t}(x_{m-1})\mright) \subset N_{p}\sigma \oplus \mleft(T_{p}M \cap \mleft(\Span(x_{j})_{j = 1}^{m-1}\mright)^{\perp}\mright),
\end{equation*}
it follows that 
\begin{equation*}
d-1 \leq \dim \Span \mleft(\mathit{II}(x_{0}, x_{1}), \dotsc, \mathit{II}(x_{0}, x_{m-1})\mright) \leq d,
\end{equation*}
from which one concludes that 
\begin{equation*}
d-1 \leq \dim \Span \mleft(\mathit{II}(x_{0}, x_{0}), \dotsc, \mathit{II}(x_{0}, x_{m-1})\mright) \leq d+1.
\end{equation*}
\end{proof}

\section{The striction submanifold}\label{TheStrictionSubmanifold}
The purpose of this section is to extend the notion of line of striction of a ruled surface in $\mathbb{R}^{3}$ to arbitrary dimension and codimension; see, e.g., \cite[section~3-5]{docarmo2016} for the classical theory. 

To begin with, let us assume that $\sigma$ is a ruled submanifold. The classical line of striction is defined under the assumption that the surface $(t,u) \mapsto \gamma(t) + u X(t)$ is noncylindrical, i.e., that $\dot{X}$ never vanishes. Here we shall require that, for all $t \in I$,
\begin{equation*}
	\begin{cases}
\dim \Span( \rho_{t}X_{1}(t), \dotsc, \rho_{t}X_{m-1}(t)) = d > 0,\\
\dim \Span(\rho_{t}X_{m -d}(t), \dotsc, \rho_{t}X_{m-1}(t)) = d;
\end{cases}
\end{equation*}
in other words, that both distributions $\mathcal{D} =\Span(X_{j})_{j=1}^{m-1}$ and $\Span(X_{h})_{h=m-d}^{m-1}$ are of degree $d >0$. Note that, by Remark~\ref{DegreeInequality}, we have the inequality
\begin{equation*}
0 < d \leq \min(m-1, n+1).
\end{equation*}

In this setting, we are going  to search for an $(m-d)$-dimensional submanifold $\beta = \beta(t, u^{1}, \dotsc, u^{m-d-1})$, lying in the image of $\sigma$ as a graph over $I \times \mathbb{R}^{m-d-1}$, such that $\langle \partial \beta / \partial t, \rho X_{h} \rangle= 0$ for all $h = m-d, \dotsc, m-1$; such submanifold is called the \textit{striction submanifold of $\sigma$}. Note that $\beta$ is contained in the image of $\sigma$ as a graph over $I \times \mathbb{R}^{m-d-1}$ if and only if there are smooth functions $u^{m-d}, \dotsc, u^{m-1}$ such that 
\begin{align*}
\beta(t, u^{1}, \dotsc, u^{m-d-1}) &= \gamma(t) + u^{1} X_{1}(t) + \dotsb + u^{m-d-1} X_{m-d-1}(t)\\
& \quad + u^{m-d}(t, u^{1}, \dotsc, u^{m-d-1}) X_{m-d}(t) + \dotsb\\
& \quad + u^{m-1}(t, u^{1}, \dotsc, u^{m-d-1}) X_{m-1}(t).
\end{align*}

We thus compute
\begin{align*}
\frac{\partial \beta}{\partial t} =	\dot{\beta} &= \dot{\gamma} + u^{1} \dot{X}_{1} + \dotsb + u^{m-d-1} \dot{X}_{m-d-1}\\
	& \quad + \dot{u}^{m-d} X_{m-d} + u^{m-d} \dot{X}_{m-d} + \dotsb\\
	& \quad + \dot{u}^{m-1} X_{m-1} + u^{m-1} \dot{X}_{m-1},
\end{align*}
which implies that
\begin{equation*}
\langle	\dot{\beta}, \rho X_{h} \rangle = \langle \dot{\gamma}, \rho X_{h} \rangle + u^{1} \langle \dot{X}_{1}, \rho X_{h} \rangle + \dotsb + u^{m-1} \langle \dot{X}_{m-1}, \rho X_{h} \rangle.
\end{equation*}

Letting 
\begin{equation*}
A =
\begin{pmatrix}
\langle \dot{X}_{m-d},  \rho X_{m-d} \rangle & \dots & \langle \dot{X}_{m-1},  \rho X_{m-d} \rangle \\
\vdots & \ddots & \vdots\\
\langle \dot{X}_{m-d},  \rho X_{m-1} \rangle & \dots & \langle \dot{X}_{m-1},  \rho X_{m-1} \rangle
\end{pmatrix} \in \mathbb{R}^{d \times d}
\end{equation*}
and
\begin{equation*}
b =  -
\begin{pmatrix}
\langle \dot{\gamma}, \rho X_{m-d} \rangle + u^{1} \langle \dot{X}_{1}, \rho X_{m-d} \rangle + \dotsb + u^{m-d-1} \langle \dot{X}_{m-d-1}, \rho X_{m-d} \rangle\\
\vdots\\
\langle \dot{\gamma}, \rho X_{m-1} \rangle + u^{1} \langle \dot{X}_{1}, \rho X_{m-1} \rangle + \dotsb + u^{m-d-1} \langle \dot{X}_{m-d-1}, \rho X_{m-1} \rangle
\end{pmatrix} \in \mathbb{R}^{d \times 1},
\end{equation*}
where $A = A(t)$ and $b = b(t, u^{1}, \dotsc, u^{m-d-1})$, our problem is equivalent to the linear system of equations
\begin{equation}\label{LinSystem}
A
\begin{pmatrix}
u^{m-d}\\
\vdots\\
u^{m-1}
\end{pmatrix}
= b
\end{equation}
in the $d$ unknowns $u^{m-d}, \dotsc, u^{m-1}$. By assumption, the matrix $A$ has full rank, and so $A^{-1}b$ is the unique solution.

It remains to show that the corresponding solution $\beta$ is well-defined, meaning that its image does not depend on the choice of the curve $\gamma$. To this end, suppose that, for another choice of $\gamma$, the striction submanifold is different from the one defined by $A^{-1}b$. Then there are two striction submanifolds, say $\beta_{1}$ and $\beta_{2}$. It follows that there exist two $d$-tuples $(u_{1}^{m-d}, \dotsc, u_{1}^{m-1})$, $(u_{2}^{m-d}, \dotsc, u_{2}^{m-1})$ of functions such that
\begin{align*}
	\beta_{\ell}(t, u^{1}, \dotsc, u^{m-d-1}) &= \gamma(t) + u^{1} X_{1}(t) + \dotsb + u^{m-d-1} X_{m-d-1}(t)\\
	& \quad + u_{\ell}^{m-d}(t, u^{1}, \dotsc, u_{\ell}^{m-d-1}) X_{m-d}(t) + \dotsb\\
	& \quad + u_{\ell}^{m-1}(t, u^{1}, \dotsc, u_{\ell}^{m-d-1}) X_{m-1}(t),
\end{align*}
with $\ell =1,2$. This implies that the solution of \eqref{LinSystem} is not unique, which is a contradiction.


Having verified that the striction submanifold is well-defined, we explain its significance by proving the following proposition.

\begin{proposition}\label{StrictionSubmanifoldPROP}
The only singular points of $\sigma$, if any, are along its striction submanifold. In particular, $\sigma$ is singular if and only if 
\begin{equation*}
\dot{\beta}(t, u^{1}, \dotsc, u^{m-d-1}) \wedge X_{1}(t) \wedge \dotsb \wedge X_{m-1}(t) = 0
\end{equation*}
for some $(t, u^{1}, \dotsc, u^{m-d-1}) \in I \times \mathbb{R}^{m-d-1}$.
\end{proposition}

\begin{proof}
First of all, we express $\sigma$ in terms of the striction submanifold:
\begin{equation*}
\sigma(t, u^{1}, \dotsc, u^{m-1}) = \beta(t, u^{1}, \dotsc, u^{m-d-1}) + u^{m-d}X_{m-d}(t) + \dotsb + u^{m-1}X_{m-1}(t).
\end{equation*}
To find the singularities of $\sigma$, we compute its partial derivatives, take their wedge product $Z$, and then set it equal to zero.

The partial derivatives of $\sigma$ with respect to $t$ and $u^{j}$ are
\begin{equation*}
\frac{\partial \sigma}{\partial t} = \dot{\beta} + u^{m-d}\dot{X}_{m-d} + \dotsb + u^{m-1}\dot{X}_{m-1}
\end{equation*}
and
\begin{align*}
\frac{\partial \sigma}{\partial u^{j}} =
\begin{cases}
X_{j} + \dfrac{\partial u^{m-d}}{\partial u^{j}} X_{m-d} + \dotsb + \dfrac{\partial u^{m-1}}{\partial u^{j}} X_{m-1} \quad &\text{if } j \in \{1, \dotsc, m-d-1\},\\
X_{j} \quad &\text{if } j \in \{m-d, \dotsc, m-1\},
\end{cases}
\end{align*}
respectively. Hence, computing their wedge product gives
\begin{align*}
Z &= \mleft(\dot{\beta} + u^{m-d}\dot{X}_{m-d} + \dotsb + u^{m-1}\dot{X}_{m-1} \mright)\\
& \quad \wedge \mleft(X_{1} + \dfrac{\partial u^{m-d}}{\partial u^{1}} X_{m-d} + \dotsb + \dfrac{\partial u^{m-1}}{\partial u^{1}} X_{m-1}\mright) \wedge \dotsb\\
& \quad \wedge \mleft( X_{m-d-1} + \dfrac{\partial u^{m-d}}{\partial u^{m-d-1}} X_{m-d} + \dotsb + \dfrac{\partial u^{m-1}}{\partial u^{m-d-1}} X_{m-1}\mright) \wedge X_{m-d} \wedge \dotsb \wedge X_{m-1}\\
&= \mleft(\dot{\beta} + u^{m-d}\dot{X}_{m-d} + \dotsb + u^{m-1}\dot{X}_{m-1}\mright) \wedge X_{1} \wedge \dotsb \wedge X_{m-1}.
\end{align*}

We claim that $Z = 0$ exactly when
\begin{equation*}
\begin{cases}
u^{m-d} = \dotsb = u^{m-1} = 0,\\
\dot{\beta} \wedge X_{1} \wedge \dotsb \wedge X_{m-1} = 0.
\end{cases}
\end{equation*}
To verify the claim, note that $Z(t, u^{1}, \dotsc, u^{m-d-1}) = 0$ if and only if
\begin{align*}
&\dot{\beta}(t, u^{1}, \dotsc, u^{m-d-1}) + u^{m-d} \dot{X}_{m-d}(t) + \dotsb  + u^{m-1} \dot{X}_{m-1}(t)\\
&\qquad = a_{1} X_{1}(t) + \dotsb + a_{m-1} X_{m-1}(t)
\end{align*}
for some scalars $a_{1}, \dotsc, a_{m-1}$, i.e.,
\begin{align}\label{SingularityCondition}
\begin{split}
\dot{\beta}(t, u^{1}, \dotsc, u^{m-d-1}) &= a_{1} X_{1}(t) + \dotsb + a_{m-1} X_{m-1}(t)\\
& \quad - u^{m-d} \dot{X}_{m-d}(t) - \dotsb  - u^{m-1} \dot{X}_{m-1}(t).
\end{split}
\end{align}
On the other hand, by definition of $\beta$, we have $\langle \dot{\beta}, \rho X_{m-d}\rangle = \dotsb = \langle \dot{\beta}, \rho X_{m-1}\rangle =0$. Hence \eqref{SingularityCondition} implies
\begin{equation*}
\begin{cases}
	u^{m-d} \langle \dot{X}_{m-d}(t), \rho_{t} X_{m-d}(t)\rangle + \dotsb + u^{m-1} \langle \dot{X}_{m-1}(t), \rho_{t} X_{m-d}(t)\rangle =0,\\
	\vdots\\
	u^{m-d} \langle \dot{X}_{m-d}(t), \rho_{t} X_{m-1}(t)\rangle + \dotsb + u^{m-1} \langle \dot{X}_{m-1}(t), \rho_{t} X_{m-1}(t)\rangle =0.
\end{cases}
\end{equation*}
This defines a homogenous linear system of full rank, which has unique (trivial) solution; thus $u^{m-d} = \dotsb = u^{m-1} = 0$.
\end{proof}

\begin{remark} \label{EquivalentCondition}
Suppose that $\rho_{t} X_{j}(t)  \neq 0$. Since the vectors $\dot{\beta}(t, u^{1},\dotsc, u^{m-d-1})$, $X_{1}(t), \dotsc, X_{m-1}(t)$ are in the orthogonal complement of $\rho_{t} X_{j}(t)$, the condition
\begin{equation*}
\dot{\beta}(t, u^{1},\dotsc, u^{m-d-1}) \wedge X_{1}(t) \wedge \dotsb \wedge X_{m-1}(t) = 0
\end{equation*}
is equivalent to
\begin{equation*}
\dot{X}_{j}(t) \wedge \dot{\beta}(t, u^{1},\dotsc, u^{m-d-1}) \wedge X_{1}(t) \wedge \dotsb \wedge X_{m-1}(t) = 0.
\end{equation*}
\end{remark}

\begin{remark}\label{SingularityRMK}
The proof of Proposition~\ref{StrictionSubmanifoldPROP} shows that the dimension of the image of $\diff \sigma$ at a singular point is $m-1$. Since, whenever $u^{m-d} = \dotsb = u^{m-1} = 0$,
\begin{equation*}
Z = \dot{\beta} \wedge \dfrac{\partial \beta}{\partial u^{1}} \wedge \dotsb \wedge \dfrac{\partial \beta}{\partial u^{m-d-1}} \wedge X_{m-d} \wedge \dotsb \wedge X_{m-1},
\end{equation*}
it follows that the image of $\diff \beta_{(t,u)}$ has dimension at least equal to $m-d-1$ for all $(t,u) \in I \times \mathbb{R}^{m-d-1}$.
\end{remark}

\section{Rank-one submanifolds}\label{Rank-oneSubmanifolds}
Here we study an important subclass of ruled submanifolds, called \emph{rank-one}.

\begin{definition}\label{RankOneDEF}
Let $\alpha$ be an $m$-dimensional submanifold of $\mathbb{R}^{m+n}$. The \textit{relative nullity index} of $\alpha$ at a regular point $q$ is the dimension of the kernel of the second fundamental form of $\alpha$ at $q$~\cite{chern1952}. We say that $\alpha$ is a \textit{rank-one} submanifold if it is ruled, and the relative nullity index at $q$ is equal to $m-1$ for all regular points $q \in \alpha$.
\end{definition}

\begin{remark}
The kernel of the second fundamental form of $\alpha$, also known as the \emph{relative nullity distribution}, coincides with the kernel of the differential of the Gauss map $\alpha \to G(m, \mathbb{R}^{m+n})$~\cite[Exercise~1.24]{dajczer2019}.
\end{remark}

\begin{remark}
Requiring that $\alpha$ be ruled in Definition~\ref{RankOneDEF} is only done for ease of presentation, as we previously defined ruled submanifolds as foliated by \emph{complete} leaves. 
\end{remark}

The next result characterizes rank-one submanifolds in two different ways.

\begin{theorem}[\cite{ushakov1999}, {\cite[Lemma~3.1]{hartman1965}}]\label{RankOneCharacterization}
If $\sigma$ is a ruled submanifold without planar points, then the following statements are equivalent:
\begin{enumerate}[font=\upshape]
\item $\sigma$ is rank-one.	
\item The induced metric on $\sigma$ is flat.
\item All tangent spaces along (the regular points of) any fixed ruling can be canonically identified with the same linear subspace of $\mathbb{R}^{m+n}$.
\end{enumerate}
\end{theorem}


According to Theorem~\ref{RankOneCharacterization}, a ruled submanifold $\sigma$ that is free of planar points is rank-one exactly when the span of its partial derivatives is independent of $u^{1}, \dotsc, u^{m-1}$. As explained in \cite[section~3]{raffaelli202x}, this condition is equivalent to $\dot{X}_{j}$ being tangent to $\sigma$ for all $j = 1, \dotsc, m-1$.

\begin{lemma}[{\cite[Corollary~3.8]{raffaelli202x}}]\label{DevelopabilityCondition}
If $\sigma \colon (t,u) \mapsto \gamma(t) + u^{j}X_{j}(t)$ is rank-one, then the following system of equations holds:
\begin{equation}\label{DevelopabilitySystem}
	\begin{cases}
\dot{X}_{1} \wedge \dot{\gamma} \wedge X_{1} \wedge \dotsb \wedge X_{m-1} = 0, \\
\vdots\\
 \dot{X}_{m-1} \wedge \dot{\gamma} \wedge X_{1} \wedge \dotsb \wedge X_{m-1} = 0.
	\end{cases}
\end{equation}
Conversely, suppose that $\sigma$ has no planar point and is regular along $\gamma$. If \eqref{DevelopabilitySystem} holds, then $\sigma$ is rank-one.
\end{lemma}

This lemma implies that a noncylindrical rank-one submanifold is necessarily of degree one, because \eqref{DevelopabilitySystem} forces $\rho_{t} X_{j}(t)$ to lie in the orthogonal complement of $\mathcal{D}_{t} = \Span(X_{j}(t))_{j=1}^{m-1}$ in the tangent space of $\sigma$, which is a subspace of dimension one. 

A further consequence of Lemma~\ref{DevelopabilityCondition} is that noncylindrical rank-one submanifolds necessarily have singularities.

\begin{proposition}\label{RankOneSingularPROP}
If $\sigma$ is rank-one and noncylindrical, then all points of the striction hypersurface are singular. Conversely, if $\sigma$ is a ruled submanifold of degree one that is singular at $\beta(t, u^{1}, \dotsc, u^{m-2})$ for all $t \in I$ and some $(u^{1}, \dotsc, u^{m-2}) \in \mathbb{R}^{m-2}$, then it is rank-one.
\end{proposition}
\begin{proof}
Once and for all, suppose that $d = 1$. Then
\begin{align}\label{StrictionHypersurface}
	\begin{split}
	\dot{\beta}(t, u^{1}, \dotsc, u^{m-2}) &= \dot{\gamma}(t) + u^{1} \dot{X}_{1}(t) + \dotsb + u^{m-2} \dot{X}_{m-2}(t)\\
	& \quad + \dot{u}^{m-1}(t, u^{1}, \dotsc, u^{m-2}) X_{m-1}(t)\\
	& \quad + u^{m-1}(t, u^{1}, \dotsc, u^{m-2}) \dot{X}_{m-1}(t).
\end{split}
\end{align}
Thus, when $\sigma$ is rank-one,
\begin{equation}\label{SingularityCondition2}
\dot{X}_{j}(t) \wedge \dot{\beta}(t, u^{1}, \dotsc, u^{m-2}) \wedge X_{1}(t) \wedge \dotsb \wedge X_{m-1}(t) = 0,
\end{equation}
which, by Remark \ref{EquivalentCondition}, implies
\begin{equation*}
\dot{\beta} \wedge X_{1} \wedge \dotsb \wedge X_{m-1} = 0.
\end{equation*} 

Conversely, if $\sigma$ is singular at $\beta(t, u^{1}, \dotsc, u^{m-2})$, then \eqref{SingularityCondition2} holds. Substituting \eqref{StrictionHypersurface}, we obtain
\begin{equation*}
\dot{X}_{j}(t) \wedge \dot{\gamma}(t) \wedge X_{1}(t)\wedge \dotsb \wedge X_{m-1}(t) = 0,
\end{equation*}
as desired.
\end{proof}

\section{Proof of Theorem \ref{MainTHM}} \label{ProofMainTHM}
We may now finalize the proof of our main result, Theorem~\ref{MainTHM} in the introduction.

Suppose that $\sigma$ is a noncylindrical rank-one submanifold. Then
\begin{equation*}
\sigma(t, u^{1}, \dotsc, u^{m-1}) = \beta(t, u^{1}, \dotsc, u^{m-2}) + u^{m-1}X_{m-1}(t)
\end{equation*}
and, by Proposition~\ref{RankOneSingularPROP}, all points of the striction hypersurface $\beta$ are singular points of $\sigma$; besides, by Remark~\ref{SingularityRMK}, $\dim d\sigma = m-1$ along $\beta$.

We shall distinguish two cases:
\begin{enumerate}
\item The striction hypersurface $\beta$ is regular. Then $\sigma$ is a tangent submanifold of $\beta$.
\item The striction hypersurface is singular everywhere, i.e., $\dim d\beta_{(t,u)} = m-2$ for all $(t,u) \in I \times \mathbb{R}^{m-2}$. Then $\sigma$ is a conical submanifold.
\end{enumerate}

Conversely, if $\sigma$ is a conical or a tangent submanifold of degree one, then it is ruled and singular. Hence, by Proposition~\ref{RankOneSingularPROP}, it is rank-one.

\section*{Acknowledgments}
The author thanks Alexander Borisenko and the reference librarians both at Saint Petersburg State University and at TU Wien for valuable help in accessing reference~\cite{ushakov1993}. Thanks also to Ivan Izmestiev and Irina Markina for help in translating parts of the same reference to English; and to Joseph O'Rourke, Ruy Tojeiro, and an anonymous referee for valuable comments on the manuscript.

\bibliographystyle{amsplain}
\bibliography{biblio}
\end{document}